\documentclass[12pt,a4paper,reqno]{amsart}
\usepackage{amsthm}
\usepackage{amsfonts}
\usepackage{amsmath}
\usepackage{amssymb}
\usepackage{mathrsfs}
\usepackage{stmaryrd}
\usepackage{bm}
\usepackage[top=1.6in, bottom=1.3in, left=1.3in, right=1.3in]{geometry}
\usepackage[colorlinks,linkcolor=blue,citecolor=blue,urlcolor=blue,hypertexnames=true]{hyperref}
\usepackage{enumitem}
\setlist[enumerate]{label=(\roman*),labelindent=1em,itemsep=0.5em,topsep=0.5em}

\title[G.C.D.\ of shifted primes and Fibonacci numbers]
{Greatest common divisors of shifted primes and Fibonacci numbers}

\author[A.~Jha]{Abhishek Jha}
\address{\parbox{\linewidth}{
Indraprastha Institute of Information Technology,\\
Okhla Industrial Estate, Phase-3, New Delhi, India}}
\email{abhishek20553@iiitd.ac.in}

\author[C.~Sanna]{Carlo Sanna}
\address{\parbox{\linewidth}{
Department of Mathematical Sciences, Politecnico di Torino\\
Corso Duca degli Abruzzi 24, 10129 Torino, Italy\\[-8pt]}}
\email{carlo.sanna.dev@gmail.com}



\theoremstyle{plain}
\newtheorem{theorem}{Theorem}[section]
\newtheorem{lemma}[theorem]{Lemma}
\newtheorem{corollary}[theorem]{Corollary}
\newtheorem{proposition}[theorem]{Proposition}
\newtheorem{remark}[theorem]{Remark}

\DeclareMathOperator*{\lcm}{lcm}

\begin{document}

\maketitle

\begin{abstract}
Let $(F_n)$ be the sequence of Fibonacci numbers and, for each positive integer $k$, let $\mathcal{P}_k$ be the set of primes $p$ such that $\gcd(p - 1, F_{p - 1}) = k$. 
We prove that the relative density $\bm{r}(\mathcal{P}_k)$ of $\mathcal{P}_k$ exists, and we give a formula for $\bm{r}(\mathcal{P}_k)$ in terms of an absolutely convergent series.
Furthermore, we give an effective criterion to establish if a given $k$ satisfies $\bm{r}(\mathcal{P}_k) > 0$, and we provide upper and lower bounds for the counting function of the set of such $k$'s.

As an application of our results, we give a new proof of a lower bound for the counting function of the set of integers of the form $\gcd(n, F_n)$, for some positive integer $n$.
Our proof is more elementary than the previous one given by Leonetti and Sanna, which relies on a result of Cubre and Rouse.

\end{abstract}

\section{Introduction}

Let $(u_n)$ be a non-degenerate linear recurrence with integral values.
Several authors studied the arithmetic relations between $u_n$ and $n$.
For instance, under the mild hypothesis that the characteristic polynomial of $(u_n)$ has only simple roots, Alba~Gonz\'{a}lez, Luca, Pomerance, and Shparlinski~\cite{MR2928495} studied the set of positive integers $n$ such that $u_n$ is divisible by $n$.
The same set was also studied by Andr\'e-Jeannin~\cite{MR1131414}, Luca and Tron~\cite{MR3409327}, Sanna~\cite{MR3606950}, and Somer~\cite{MR1271392}, in the special case in which $(u_n)$ is a Lucas sequence.
Furthermore, Sanna~\cite{MR3935356} studied the set of natural numbers $n$ such that $\gcd(n, u_n) = 1$ (see~\cite{MR3896876} for a generalization, and~\cite{MR4159096} for a survey on g.c.d.'s of linear recurrences).
Similar problems, with $(u_n)$ replaced by an elliptic divisibility sequence or by the orbit of a polynomial map, were also studied~\cite{MR3690240,MR4067105,MR2928585,JhaPreprint,MR4017936,MR2747036}.

Let $(F_n)$ be the linear recurrence of Fibonacci numbers, which is defined as usual by $F_1 = F_2 = 1$ and $F_{n+2} = F_{n+1} + F_n$ for all positive integers $n$. 
For every positive integer $k$, define the following set of natural numbers
\begin{equation*}
    \mathcal{A}_k := \{n \geq 1 : \gcd(n, F_n) = k\} ,
\end{equation*}
Recall that the natural density $\bm{d}(\mathcal{S})$ of a set of positive integers $\mathcal{S}$ is defined as the limit of the ratio $\#\big(\mathcal{S} \cap [1, x]\big) / x$ as $x \to +\infty$, whenever this limit exists.
Sanna and Tron~\cite{MR3800663} proved that each $\mathcal{A}_k$ has a natural density, which can be written as an infinite series, and they provided an effective criterion to determine if such density is positive.

We consider similar results but for the set of \emph{shifted primes} $p - 1$.
(Throughout, we reserve the letter $p$ for prime numbers.)
Shifted primes already make their appearance in relation to Fibonacci numbers.
For instance, it is well known that $p$ divides $F_{p - 1}$ for every prime number $p \equiv \pm 1 \pmod 5$.
For each integer $k \geq 1$, define the following set of prime numbers
\begin{equation*}
    \mathcal{P}_k := \{p : \gcd(p-1, F_{p-1}) = k\} .
\end{equation*}
Recall that the \emph{relative density} $\mathbf{r}(\mathcal{P})$ of a set of prime numbers $\mathcal{P}$ is defined as the limit of the ratio $\#(\mathcal{P} \cap [1, x]) / \pi(x)$ as $x \to +\infty$, whenever this limit exists, where $\pi(x)$ denotes the number of primes not exceeding $x$.
Let $z(m)$ denote the \emph{rank of appearance}, or \emph{entry point}, of a positive integer $m$ in the sequence of Fibonacci numbers, that is, the smallest positive integer $n$ such that $m$ divides $F_{n}$. 
It is well known that $z(m)$ exists. 
Also, let $\ell(m)=\lcm\!\big(m,z(m)\big)$.

Our first result establishes the existence of the relative density of $\mathcal{P}_k$ and provides a criterion to check if such a density is positive.

\begin{theorem}\label{thm:existence}
    For each positive integer $k$, the relative density of $\mathcal{P}_k$ exists. 
    Moreover, if $\gcd\!\big(\ell(k),F_{\ell(k)}\big) \neq k$, or if $2 \nmid \ell(k)$ and $\ell(pk) = 2\,\ell(k)$ for some prime number $p$ with $p \nmid k$, then $\mathcal{P}_k \subseteq \{2\}$.
    Otherwise, we have that $\bm{r}(\mathcal{P}_k) > 0$.
\end{theorem}

For instance, $k = 17$ is the smallest positive integer such that $\bm{d}(\mathcal{A}_k) > 0$, since $\gcd\!\big(\ell(k),F_{\ell(k)}\big) = k$ (see Lemma~\ref{lem:Ak-characterization} below) but $\bm{r}(\mathcal{P}_k) = 0$, since $\ell(k) = 153$ is odd and $\ell(pk) = 2\,\ell(k)$ for $p = 2$.

Our second result gives an explicit expression for the relative density of $\mathcal{P}_k$ in terms of an absolutely convergent series.

\begin{theorem}\label{thm:series}
    For each positive integer $k$, the relative density of $\mathcal{P}_k$ is
    \begin{equation*}
        \mathbf{r}(\mathcal{P}_k) = \sum_{d \,=\, 1}^\infty \frac{\mu(d)}{\varphi(\ell(dk))} ,
    \end{equation*}
    where $\mu$ is the M\"obius function, $\varphi$ is the Euler totient function, and the series converges absolutely.
\end{theorem}

Leonetti and Sanna~\cite{MR3859754} proved the following upper and lower bounds for the counting function of the set $\mathcal{A} := \{\gcd(n, F_n) : n \geq 1\}$.

\begin{theorem}\label{thm:A-bounds}
We have
\begin{equation}\label{equ:A-bounds}
    \frac{x}{\log x} \ll \#\big(\mathcal{A} \cap [1, x] \big) = o(x) ,
\end{equation}
as $x \to +\infty$.
\end{theorem}

As an application of Theorem~\ref{thm:existence}, we provide an alternative proof of the lower bound in~\eqref{equ:A-bounds}.
We remark that our proof uses quite elementary methods, while Leonetti and Sanna's proof relies on a result of Cubre and Rouse~\cite{MR3251719}, which in turn is proved by Galois theory and Chebotarev's density theorem.

Let $\mathcal{K}$ be the set of positive integers $k$ such that $\bm{r}(\mathcal{P}_k) > 0$.
We have the following upper and lower bounds for the counting function of $\mathcal{K}$.

\begin{proposition}\label{prop:K-bounds}
We have
\begin{equation*}
    \frac{x}{\log x} \ll \#\big(\mathcal{K} \cap [1,x]\big) = o(x) ,
\end{equation*}
as $x \to +\infty$.
\end{proposition}

We remark that both Theorem~\ref{thm:existence} and Theorem~\ref{thm:series} can be generalized to non-degenerate Lucas sequences, that is, integer sequences $(u_n)$ such that $u_1 = 1$, $u_2 = a_1$, and $u_n = a_1 u_{n - 1} + a_2 u_{n - 2}$, for every integer $n \geq 2$, where $a_1, a_2$ are non-zero relatively prime integers such that the ratio of the roots of $X^2 - a_1 X - a_2$ is not a root of unity.
We decided to focus on the sequence of Fibonacci numbers in order to simplify the exposition.

A generalization in another direction could be studying the sets of primes
\begin{equation*}
    \mathcal{P}_k^{(s)} := \{p : \gcd(p + s, F_{p + s}) = k\} ,
\end{equation*}
for integers $k \geq 1$ and $s$.

\subsection*{Acknowledgments}

C.~Sanna is a member of GNSAGA of INdAM and of CrypTO, the group of Cryptography and Number Theory of Politecnico di Torino.

\section{Preliminaries on primes in certain residue classes}

We shall need a mild generalization (Theorem~\ref{thm:primes-residue-classes} below) of a result of Leonetti and Sanna~\cite{MR3846400} on primes in certain residue classes.
First, we have to introduce some notation.
For all $x \leq y$, let $\llbracket x, y \rrbracket := [x, y] \cap \mathbb{N}$.
For vectors $\bm{x} = (x_1, \dots, x_d)$ and $\bm{y} = (y_1, \dots, y_d)$ in $\mathbb{N}^d$, let $\|\bm{x}\| := x_1 \cdots x_d$, $\llbracket \bm{x}, \bm{y} \rrbracket := \llbracket x_1, y_1 \rrbracket \times \cdots \times \llbracket x_d, y_d \rrbracket$, $\bm{x}\bm{y} := (x_1 y_1, \dots, x_d y_d)$, and $\bm{x}/\bm{y} := (x_1 / y_1, \dots, x_d / y_d)$.
Let $\bm{0}$, respectively $\bm{1}$, be the vector of $\mathbb{N}^d$ with all components equal to $0$, respectively $1$.
For every $\bm{m} = (m_1, \dots, m_d) \in \mathbb{N}^d$, write $\bm{x} \equiv \bm{y} \pmod {\bm{m}}$ if and only if $x_i \equiv y_i \pmod {m_i}$ for each $i \in \llbracket 1, d \rrbracket$, and write instead $\bm{x} \not\equiv \bm{y} \pmod {\bm{m}}$ if and only if $x_i \not\equiv y_i \pmod {m}$ for at least one $i \in \llbracket 1, d \rrbracket$.

\begin{lemma}\label{lem:multidimensional}
Let $d$ be a positive integer and let $\bm{c}_1, \dots, \bm{c}_k, \bm{d} \in \mathbb{N}^d$ be vectors such that $\bm{c}_1 \cdots \bm{c}_k \equiv \bm{0} \pmod {\bm{d}}$ and $\bm{d} \equiv \bm{0} \pmod {\bm{c}_i}$ for each $i \in \llbracket 1, k \rrbracket$. 
Then the set $\mathcal{X}$ of all $\bm{x} \in \llbracket \bm{1}, \bm{d} \rrbracket$ such that $\bm{x} \not\equiv \bm{0} \pmod {\bm{c}_i}$ for each $i \in \llbracket 1, k \rrbracket$ satisfies
\begin{equation*}
\#\mathcal{X} \geq \|\bm{d}\| \cdot \prod_{i \,=\, 1}^k \left(1 - \frac1{\|\bm{c}_i\|}\right) .
\end{equation*}
\end{lemma}
\begin{proof}
See~\cite[Lemma~2.1]{MR3846400}.
\end{proof}

For all positive integers $a_0, \dots, a_k$, let $\mathcal{Q}(a_0, \dots, a_k)$ be the set of primes $p$ such that $p \equiv 1 \pmod {a_0}$ and $p\not\equiv 1 \pmod {a_i}$ for every $i \in \llbracket 1, k \rrbracket$.

\begin{theorem}\label{thm:primes-residue-classes}
Let $a_0,\dots,a_k$ be positive integers with $a_0 \mid a_i$ for each $i \in \llbracket 1, k \rrbracket$.
Then the relative density of $\mathcal{Q} := \mathcal{Q}(a_0, \dots, a_k)$ exists and satisfies
\begin{equation}\label{equ:rP0}
\bm{r}(\mathcal{Q}) \geq \frac1{\varphi(a_0)} \prod_{i \,=\, 1}^k \left(1-\frac{\varphi(a_0)}{\varphi(a_i)}\right) .
\end{equation}
\end{theorem}
\begin{proof}
We generalize the proof of \cite[Theorem~1.2]{MR3846400}, which corresponds to the special case $a_0 = 1$.
Let $L := \lcm(a_0, \dots, a_k) = p_1^{e_1} \cdots p_d^{e_d}$  where $p_1 < \cdots < p_d$ are primes and $e_1, \dots, e_d$ are positive integers.
Also, let $\mathcal{S}$ be the set of integers $n \in [1, L]$ such that: $\gcd(n, L) = 1$, $n \equiv 1 \pmod {a_0}$, and $n \not\equiv 1 \pmod {a_i}$ for every $i \in \llbracket 1,k \rrbracket$.
By Dirichlet's theorem on primes in arithmetic progressions, we have that
\begin{align}\label{equ:rP1}
    \bm{r}(\mathcal{Q}) &= \lim_{x \,\to\, +\infty} \frac{\#(\mathcal{Q} \cap [1, x])}{\pi(x)} \nonumber \\
        &= \lim_{x \,\to\, +\infty} \sum_{n \,\in\, \mathcal{S}} \frac{\#\{p \leq x : p \equiv n \!\!\!\!\pmod L\}}{\pi(x)} = \frac{\#\mathcal{S}}{\varphi(L)} .
\end{align}
Hence, the relative density of $\mathcal{Q}$ exists.
Let us give a lower bound on $\#\mathcal{S}$.

First, assume that $8 \nmid L$.
Let $g_j$ be a primitive root modulo $p_j^{e_j}$, for each $j \in \llbracket 1,d \rrbracket$.
Note that $g_1$ exists when $p_1 = 2$ since $e_1 \leq 2$.
Put $\bm{b} := \big(\varphi(p_1^{e_1}), \dots, \varphi(p_d^{e_d})\big)$.
By the Chinese remainder theorem, each $n \in \llbracket 1, \ell \rrbracket$ with $\gcd(n, L) = 1$ is uniquely determined by a vector $\bm{y}(n) = (y_1(n), \dots, y_d(n)) \in \llbracket \bm{1}, \bm{b} \rrbracket$ such that $n \equiv g_j^{y_j(n)} \pmod {p_j^{e_j}}$ for each $j \in \llbracket 1,d \rrbracket$.
Write $a_i = p_1^{\alpha_{i,1}} \cdots p_d^{\alpha_{i,d}}$, where $\alpha_{i,1}, \dots,\alpha_{i,d} \geq 0$ are integers, and define $\bm{a}_i := \big(\varphi(p_1^{\alpha_{i,1}}), \dots, \varphi(p_d^{\alpha_{i,d}})\big)$ for each $i \in \llbracket 0,k \rrbracket$.
Also, put $\bm{c}_i = \bm{a}_i / \bm{a}_0$ for every $i \in \llbracket 0,k \rrbracket$, $\bm{d} := \bm{b} / \bm{a}_0$, and let $\mathcal{X}$ be defined as in Lemma~\ref{lem:multidimensional}.
At this point, it follows easily that $n \in \mathcal{S}$ if and only if $\bm{y}(n) \equiv \bm{0} \pmod {\bm{a}_0}$ and $\bm{y}(n) \not\equiv \bm{0} \pmod {\bm{a}_i}$ for each $i \in \llbracket 1, k\rrbracket$.
Therefore, the map $n \mapsto \bm{y}(n) / \bm{a}_0$ is a bijection $\mathcal{S} \to \mathcal{X}$ and, consequently, $\#\mathcal{S} = \#\mathcal{X}$.
Since $\|\bm{d}\| = \varphi(L) / \varphi(a_0)$, $\|\bm{c}_i\| = \varphi(a_i) / \varphi(a_0)$, $\bm{c}_1\cdots\bm{c}_k \equiv \bm{0} \pmod {\bm{d}}$, and $\bm{d} \equiv \bm{0} \pmod {\bm{c}_i}$ for each $i \in \llbracket 1,k \rrbracket$, we can apply Lemma~\ref{lem:multidimensional}, which gives a lower bound on $\#\mathcal{X}$, that is, on $\#\mathcal{S}$.
Then~\eqref{equ:rP1} and the lower bound on $\#\mathcal{S}$ yield~\eqref{equ:rP0}.

Now let us consider the case in which $8 \mid L$.
This is a bit more involved since there are no primitive roots modulo $2^e$, for every integer $e \geq 3$.
However, the previous arguments still work by changing $\bm{a}_i$ and $\bm{b}$ with
\begin{equation*}
\bm{a}_i := \big(2^{\max(0,\,\alpha_{i,1} - 1) - \max(0,\, \alpha_{i,1} - 2)}, 2^{\max(0,\, \alpha_{i,1} - 2)}, \varphi(p^{\alpha_{i,2}}), \ldots, \varphi(p^{\alpha_{i,d}})\big)
\end{equation*}
and
\begin{equation*}
\bm{b} = \big(2, 2^{e_1 - 2}, \varphi(p_2^{e_2}), \dots, \varphi(p_d^{e_d})\big) .
\end{equation*}
Then each $n \in \llbracket 1, \ell \rrbracket$ with $\gcd(n, L) = 1$ is uniquely determined by a vector $\bm{y}(n) = (y_0(n), \dots, y_d(n)) \in \llbracket \bm{1}, \bm{b} \rrbracket$ such that $n \equiv (-1)^{y_0(n)} 5^{y_1(n)} \pmod {2^{e_1}}$ and $n \equiv g_j^{y_j(n)} \pmod {p_j^{e_j}}$ for each $j \in \llbracket 2, d \rrbracket$.
The rest of the proof proceeds similarly to the previous case.
\end{proof}

For all positive integers $a_0, a_1, \dots$, let $\mathcal{Q}(a_0, a_1, \dots) := \bigcap_{k \,\geq\, 1} \mathcal{Q}(a_0, \dots, a_k)$.

\begin{corollary}\label{cor:primes-residue-classes}
If $a_0,a_1,\dots$ is a sequence of positive integers such that $a_0 \mid a_i$ for each integer $i \geq 1$ and the series $\sum_{i \,\geq\, 1} 1 / \varphi(a_i)$ converges, then the relative density of $\mathcal{Q} := \mathcal{Q}(a_0, a_1, \dots)$ exists.
Moreover, $\bm{r}(\mathcal{Q}) = 0$ if and only if there exists an integer $i \geq 1$ such that $a_i = a_0$, or $a_i = 2a_0$ and $a_0$ is odd.
In such a case, we have that $\mathcal{Q} \subseteq \{2\}$.
\end{corollary}
\begin{proof}
If there exists an integer $i \geq 1$ such that $a_i = a_0$, or $a_i = 2a_0$ and $a_0$ is odd, then it follows easily that $\mathcal{Q} \subseteq \{2\}$ and, consequently, $\bm{r}(\mathcal{Q}) = 0$.
Hence, assume that no such integer $i$ exists.
In particular, we have that $\varphi(a_0) < \varphi(a_i)$ for every integer $i \geq 1$. 
From Theorem~\ref{thm:primes-residue-classes} we know that, for every integer $k \geq 1$, the relative density of $\mathcal{Q}_k := \mathcal{Q}(a_0, \dots, a_k)$ exists and
\begin{align*}
r := \lim_{k \,\to\, +\infty} \bm{r}(\mathcal{Q}_k) \geq \frac1{\varphi(a_0)} \prod_{i \,=\, 1}^\infty \left(1 - \frac{\varphi(a_0)}{\varphi(a_i)}\right) > 0 ,
\end{align*}
where the infinite product converges to a positive number since $\sum_{i \,\geq\, 1} 1 / \varphi(a_i)$ converges and $\varphi(a_0) / \varphi(a_i) < 1$ for every integer $i \geq 1$.
Furthermore, for each $\varepsilon > 0$ and for every sufficiently large positive integer $k = k(\varepsilon)$, we have that
\begin{align*}
&\phantom{mm}\limsup_{x \,\to\, +\infty} \left|r - \frac{\#(\mathcal{Q} \cap [1,x])}{\pi(x)}\right| < \varepsilon + \limsup_{x \,\to\, +\infty} \frac{\#\big((\mathcal{Q}_k \setminus \mathcal{Q})\cap [1, x]\big)}{\pi(x)} \\
&\leq \varepsilon + \limsup_{x \,\to\, +\infty} \frac{\#\{p \leq x : \exists j > k \text{ s.t. } p \equiv 1 \!\!\!\pmod {a_j}\}}{\pi(x)} \leq \varepsilon + \sum_{j \,>\, k} \frac1{\varphi(a_j)} < 2 \varepsilon .
\end{align*}
Therefore, the relative density of $\mathcal{Q}$ exists and, in fact, $\bm{r}(\mathcal{Q}) = r > 0$.
\end{proof}

\section{Further preliminaries}

The next lemma summarizes some basic properties of the Fibonacci numbers and the arithmetic functions $\ell$ and $z$.

\begin{lemma}\label{lem:basic}
For all positive integers $m$, $n$ and all prime numbers $p$, we have:
\begin{enumerate}
\item\label{ite:basic:fibdiv} $F_m \mid F_n$ whenever $m \mid n$.
\item\label{ite:basic:fibratio} $\gcd(F_m / F_n, F_n) \mid m / n$ whenever $n\mid m$.
\item\label{ite:basic:zeta} $m \mid F_n$ if and only if $z(m) \mid n$.
\item\label{ite:basic:zetap} $z(p) \mid p - \left(\displaystyle\frac{p}{5}\right)$ where $\left(\displaystyle\frac{p}{5}\right)$ is a Legendre symbol.
\item\label{ite:basic:gcd} $m \mid \gcd(n, F_n)$ if and only if $\ell(m) \mid n$.
\item\label{ite:basic:lcm} $\ell(\lcm(m, n)) = \lcm(\ell(m), \ell(n))$.
\item\label{ite:basic:ellp} $\ell(p) = z(p) p$ for $p \neq 5$, and $\ell(5) = 5$.
\item\label{ite:basic:upper} $\ell(n) \leq 2n^2$.
\end{enumerate}
\end{lemma}
\begin{proof}
Facts~\ref{ite:basic:fibdiv}--\ref{ite:basic:zetap} are well known (for~\ref{ite:basic:fibratio}, see~\cite[Lemma~2]{MR491445}).
Facts~\ref{ite:basic:gcd}--\ref{ite:basic:ellp} follow easily from~\ref{ite:basic:zeta} and~\ref{ite:basic:zetap} and the definition of $\ell$ (cf.~\cite[Lemma~2.1]{MR3800663}).
Finally, fact~\ref{ite:basic:upper} follows easily from the well-known inequality $z(n) \leq 2 n$~(see, e.g., \cite{MR360447}).
\end{proof}

Now we state a result to establish if $\mathcal{A}_k \neq \varnothing$ and $\bm{d}(\mathcal{A}_k) > 0$.

\begin{lemma}\label{lem:nonempty-Ak}
$\mathcal{A}_k \neq \varnothing$ if and only if $\bm{d}(\mathcal{A}_k) > 0$ if and only if $\gcd\!\big(\ell(k),F_{\ell(k)}\big)=k$, for all integers $k \geq 1$.
\end{lemma}
\begin{proof}
See~\cite[Theorem~1.3]{MR3800663}.
\end{proof}

\begin{lemma}\label{lem:Ak-characterization}
Let $k$ and $n$ be positive integers.
Suppose that $\mathcal{A}_k \neq \varnothing$.
Then $n \in \mathcal{A}_k$ if and only if $\ell(k) \mid n$ and $m \nmid n$ for every
\begin{equation*}
m \in \big\{p\,\ell(k) : p \mid k\big\} \cup \{\ell(pk) : p \nmid k\} .
\end{equation*}
\end{lemma}
\begin{proof}
See~\cite[Lemma~3.1]{MR3800663}.
\end{proof}

We need some upper bounds for series involving $\ell(n)$.

\begin{lemma}\label{lem:bound-sum-ell}
We have
\begin{equation*}
\sum_{n \,>\, y} \frac1{\ell(n)} <
\exp\!\left(-\delta (\log y)^{1/2} (\log \log y)^{1/2} \right) ,
\end{equation*}
for all $\delta \in \big(0, 1/\!\sqrt{6}\;\!\big)$ and $y \gg_\delta 1$.
\end{lemma}
\begin{proof}
See~\cite[Proposition~1.4]{MR3983305}.
\end{proof}

\begin{lemma}\label{lem:bound-phi-ell}
We have
\begin{equation*}
\sum_{n \,>\, y}\frac{1}{\varphi\big(\ell(n)\big)} 
\ll \frac{\log \log y}{\exp\!\left(\delta (\log y)^{1/2} (\log \log y)^{1/2} \right)} ,
\end{equation*} 
for all $\delta \in \big(0, 1/\!\sqrt{6}\;\!\big)$ and $y \gg_\delta 1$.
\end{lemma}
\begin{proof}
From Lemma~\ref{lem:bound-sum-ell} it follows that
\begin{equation*}
S(t) := \sum_{n \,\geq\, t} \frac1{\ell(n)}
< f(t) := \exp\!\left(-\delta (\log t)^{1/2} (\log \log t)^{1/2} \right) ,
\end{equation*}
for all $t \gg_\delta 1$.
By partial summation, we obtain that
\begin{align*}\label{equi}
\sum_{n \,\geq\, y}\frac{\log{\log{n}}}{\ell(n)} &= S(y)\log \log y +\int_{y}^{+\infty} \frac{S(t)}{t \log t}\,\mathrm{d} t \\
&< f(y)\log\log y + \int_y^{+\infty} \frac{f(t)}{t \log t}\,\mathrm{d} t \\
&\ll_\delta f(y)\log\log y - \int_y^{+\infty} f^\prime(t)\,\mathrm{d} t \\
&\ll f(y)\log\log y .
\end{align*}
Then, since $\varphi(n)\gg n/\!\log{\log{n}}$ (see, e.g., \cite[Chapter I.5, Theorem 4]{MR3363366}) and $\ell(n)\leq 2n^2$ (Lemma \ref{lem:basic}\ref{ite:basic:upper}) for all positive integers $n$, we have that 
\begin{equation*}
\sum_{n \,>\, y}\frac{1}{\varphi\big(\ell(n)\big)} \ll \sum_{n \,>\,
y}\frac{\log{\log{n}}}{\ell(n)} \ll f(y) \log \log y .
\end{equation*}
The claim follows.
\end{proof}

For every $x > 0$ and for all integers $a$ and $b$, let $\pi(x;b,a)$ be the number of primes $p \leq x$ such that $p \equiv a \pmod b$, and put also
\begin{equation*}
\Delta(x; b, a) := \pi(x;b,a) - \frac{\pi(x)}{\varphi(b)} .
\end{equation*}
We need the following bounds for $\Delta(x;b,a)$.
\begin{theorem}[Siegel--Walfisz]\label{thm:Siegel-Walfisz}
For every $A > 0$, we have,
\begin{equation*}
\Delta(x; b, a) \ll \frac{x}{(\log x)^{A}} ,
\end{equation*}
for all $x \gg_A 1$ and for all relatively prime positive integers $a,b$ with $b \leq (\log x)^A$.
\end{theorem}
\begin{proof}
See~\cite[Corollary 5.29]{MR2061214}.
\end{proof}

\begin{lemma}\label{lem:corollary-brun-titchmarsh}
Let $\varepsilon > 0$.
Then we have that
\begin{equation*}
\Delta(x; b, a) \ll_\varepsilon \frac{x}{\varphi(b) \log x} ,
\end{equation*}
for all $x \geq 2$ and for all relatively prime positive integers $a,b$ with $b \leq x^{1-\varepsilon}$.
\end{lemma}
\begin{proof}
From the Brun--Titchmarsh theorem~\cite[Theorem~9]{MR3363366} we know that
\begin{equation*}
\pi(x;b,a) \ll \frac{x}{\varphi(b) \log (x/b)} ,
\end{equation*}
for all $b < x$.
Hence, the condition $b \leq x^{1-\varepsilon}$ and the upper bound $\pi(x) \ll x / \log x$ yield that
\begin{equation*}
\Delta(x; b, a) \ll \pi(x;b,a) + \frac{\pi(x)}{\varphi(b)} \ll \frac{x}{\varphi(b)\log (x / b)} + \frac{x}{\varphi(b) \log x} \ll_\varepsilon \frac{x}{\varphi(b) \log x} ,
\end{equation*}
as desired.
\end{proof}

\section{Proof of Theorem~\ref{thm:existence}}

Let $k$ be a positive integer.
If $\mathcal{P}_k = \varnothing$ then, obviously, the relative density of $\mathcal{P}_k$ exists and is equal to zero.
Hence, suppose that $\mathcal{P}_k \neq \varnothing$.
In~particular, $\mathcal{A}_k \neq \varnothing$, since $p - 1 \in \mathcal{A}_k$ for every $p \in \mathcal{P}_k$.
Therefore, by Lemma~\ref{lem:nonempty-Ak}, we have that $\gcd\!\big(\ell(k), F_{\ell(k)}\big) = k$.
Recall the definition of $\mathcal{Q}(a_0, a_1, \dots)$ given before Corollary~\ref{cor:primes-residue-classes}.
Define the sequence $\mathcal{M}_k = m_0, m_1, \dots$ where $m_0 < m_1 < \dots$ are all the elements of
\begin{equation*}
\big\{\ell(k)\big\} \cup \big\{p\,\ell(k) : p \mid k \big\} \cup \big\{\ell(pk) : p \nmid k \big\} .
\end{equation*}
Then, from Lemma~\ref{lem:Ak-characterization} and the definition of $\mathcal{Q}(\mathcal{M}_k)$, it follows that $\mathcal{P}_k = \mathcal{Q}(\mathcal{M}_k)$.
Furthermore, by Lemma~\ref{lem:bound-phi-ell}, we have that
\begin{equation*}
\sum_{i \,\geq\, 0} \frac1{\varphi(m_i)} \ll_k \sum_{p} \frac1{\varphi\big(\ell(pk)\big)} \ll_k \sum_{p} \frac1{\varphi\big(\ell(p)\big)} < +\infty .
\end{equation*}
Hence, thanks to Corollary~\ref{cor:primes-residue-classes}, we get that the relative density of $\mathcal{P}_k$ exists and, in particular, $\bm{r}(\mathcal{P}_k) = 0$ if and only if $\mathcal{P}_k \subseteq \{2\}$ if and only if there exists an integer $i \geq 1$ such that $m_i = m_0$, or $m_i = 2m_0$ and $m_0$ is odd.
The first case is impossible, since the sequence $\mathcal{M}_k$ is increasing.
The second case is equivalent to $2 \nmid \ell(k)$ and either $p\, \ell(k) = 2\,\ell(k)$, for some prime number $p$ with $p \mid k$, or $\ell(pk) = 2\,\ell(k)$, for some prime number $p$ with $p \nmid k$.
In turn, since $k \mid \ell(k)$, this is equivalent to $2 \nmid \ell(k)$ and $\ell(pk) = 2\,\ell(k)$ for some prime number $p$ with $p \nmid k$.
The proof is complete.

\begin{remark}
We remark that the convergence of the series
\begin{equation*}\sum_{p} \frac1{\varphi\big(\ell(p)\big)}\end{equation*} 
admits a simpler proof than invoking Lemma~\ref{lem:bound-phi-ell} which we highlight below. \end{remark}
\begin{proof}
Note that $\ell(p) \gg p\,z(p) \gg p \log p$ due to Lemma~\ref{lem:basic}\ref{ite:basic:ellp}.
Thus, we have that 
\begin{equation*}\sum_{p} \frac1{\varphi\big(\ell(p)\big)}\ll \sum_{p}\frac{\log\log p}{p\,z(p)}\ll \sum_{p}\frac{\log \log p}{p\log p} < +\infty \end{equation*} since $\varphi(n)\gg n/\!\log{\log{n}}$ (see, e.g., \cite[Chapter I.5, Theorem 4]{MR3363366}) and $\ell(n)\leq 2n^2$ (Lemma \ref{lem:basic}\ref{ite:basic:upper}) for all positive integers $n$, and the convergence of last sum is standard .
\end{proof}

\section{Proof of Theorem~\ref{thm:series}}

For each positive integer $k$, let $\mathcal{R}_k$ be the set of prime numbers $p$ such that:

\begin{enumerate}
\item $k \mid \gcd(p-1, F_{p-1})$;
\item if $q \mid \gcd(p-1, F_{p-1})$ for some prime number $q$, then $q \mid k$.
\end{enumerate}
The essential part of the proof of Theorem~\ref{thm:series} is the following formula for the relative density of $\mathcal{R}_k$.

\begin{lemma}\label{lem:Qkdens}
For all positive integers $k$, the relative density of $\mathcal{R}_k$ exists and
\begin{equation}\label{equ:densQk}
\mathbf{r}(\mathcal{R}_k) = \sum_{(d,\,k) \,=\, 1} \frac{\mu(d)}{\varphi\big(\ell(dk)\big)} ,
\end{equation}
where the series is absolutely convergent.
\end{lemma}
\begin{proof}
For every prime $p$ and for every positive integer $d$, let us define
\begin{equation*}
\varrho(p, d) := \begin{cases} 1 & \text{ if } d \mid F_{p-1}, \\ 0 & \text{ if } d \nmid F_{p-1} .\end{cases}
\end{equation*}
Note that $\varrho$ is multiplicative in its second argument, that is,
\begin{equation*}
\varrho(p, de) = \varrho(p, d) \, \varrho(p, e)
\end{equation*}
for all primes $p$ and for all coprime positive integers $d$ and $e$.

From Lemma~\ref{lem:basic}\ref{ite:basic:gcd}, it follows easily that $p \in \mathcal{R}_k$ if and only if $p\equiv 1\pmod{\ell(k)}$ and $\varrho(p, q) = 0$ for all prime numbers $q$ dividing $p-1$ but not dividing $k$.
Therefore,
\begin{align}\label{equ:count1}
\#\big(\mathcal{R}_k \cap [1, x]\big) &= \sum_{\substack{p \,\leq\, x \\[1pt] \ell(k) \,\mid\, p-1}} \prod_{\substack{q \,\mid\, p-1 \\[1pt] q \,\nmid\, k}} \big(1 - \varrho(p, q)\big) = \sum_{\substack{p \,\leq\, x \\[1pt] \ell(k) \,\mid\, p-1}} \sum_{\substack{d \,\mid\, p-1 \\ (d,\, k) \,=\, 1}} \mu(d)\, \varrho(p, d) \nonumber \\
&= \sum_{\substack{d \,\leq\, x \\[1pt] (d,\, k) \,=\, 1}} \mu(d) \sum_{\substack{p \,\leq\, x \\[1pt] \lcm(\ell(k),\,d) \,\mid\, p-1}} \varrho(p, d) ,
\end{align}
for all $x > 0$.
Furthermore, by Lemma~\ref{lem:basic}\ref{ite:basic:zeta}, given a positive integer $d$ that is relatively prime with $k$, we have that $\varrho(p, d) = 1$ and $\lcm(d,\ell(k)) \mid p-1$ if and only if $\lcm(z(d),d, \ell(k)) \mid p-1$, which in turn is equivalent to $p-1$ being divisible by
\begin{equation*}
\lcm\!\big(\!\lcm\!\big(z(d), d\big), \ell(k)\big) = \lcm\!\big(\ell(d), \ell(k)\big) = \ell(dk) ,
\end{equation*}
where we used Lemma~\ref{lem:basic}\ref{ite:basic:lcm} and the fact that $d$ and $k$ are relatively prime.
Hence, we get that
\begin{equation}\label{equ:inner-sum}
\sum_{\substack{p \,\leq\, x \\[1pt] \lcm(\ell(k),d) \,\mid\, p-1}} \varrho(p, d) = \!\sum_{\substack{p \,\leq\, x \\[1pt] p \,\equiv\, 1 \!\!\!\!\pmod{\ell(dk)}}} 1 = \pi\big(x;\ell(dk),1\big) ,
\end{equation}
for all $x > 0$. 
Therefore, from~\eqref{equ:count1} and~\eqref{equ:inner-sum}, it follows that
\begin{equation*}
\#\big(\mathcal{R}_k \cap [1, x]\big) = \sum_{\substack{d \,\leq\, x \\[1pt] (d,\, k) \,=\, 1}} \mu(d) \, \pi\big(x;\ell(dk),1\big) ,
\end{equation*}
for all $x > 0$.
Pick any $A > 2$.
Also, set $y := x^{1/4} / \big(\!\sqrt{2} k\big)$ and $z := (\log x)^{A/2} / \big(\!\sqrt{2} k\big)$.
Then we have that
\begin{align*}
\frac{\#\big(\mathcal{R}_k \cap [1, x]\big)}{\pi(x)} = \sum_{(d,\, k) \,=\, 1} \frac{\mu(d)}{\varphi\big(\ell(dk)\big)} - E_1(x) + E_2(x) + E_3(x)+E_4(x)
\end{align*}
for all $x > 0$, where, by Lemma~\ref{lem:bound-phi-ell}, the infinite series converges absolutely, while
\begin{equation*}
E_1(x) := \sum_{\substack{d \,>\, x \\[1pt] (d,\, k) \,=\, 1}} \frac{\mu(d)}{\varphi\big(\ell(dk)\big)} ,
\end{equation*}
\begin{equation*}
E_2(x) := \frac1{\pi(x)}\sum_{\substack{d \,\leq\, z \\[1pt] (d,\, k) \,=\, 1}} \mu(d)\,\Delta\big(x;\ell(dk),1\big) ,
\end{equation*}
\begin{equation*}
E_3(x) := \frac1{\pi(x)}\sum_{\substack{z \,<\, d \,\leq\, y \\[1pt] (d,\, k) \,=\, 1}} \mu(d) \, \Delta\big(x;\ell(dk),1\big) ,
\end{equation*}
and 
\begin{equation*}
    E_4(x)=\frac1{\pi(x)}\sum_{\substack{y \,<\, d \,\leq\, x \\[1pt] (d,\, k) \,=\, 1}} \mu(d) \, \Delta\big(x;\ell(dk),1\big) ,
\end{equation*}
It remains only to prove that $E_1(x)$, $E_2(x)$, $E_3(x)$, $E_4(x)$ go to zero as $x \to +\infty$.
From Lemma~\ref{lem:bound-phi-ell} it follows that
\begin{equation*}
E_1(x) \ll \sum_{d \,>\, y} \frac1{\varphi\big(\ell(d)\big)} = o(1) ,
\end{equation*}
as $x \to +\infty$.
Note that, thanks to Lemma~\ref{lem:basic}\ref{ite:basic:upper}, if $d \leq z$ then $\ell(dk) \leq (\log x)^A$.
Hence, from Theorem~\ref{thm:Siegel-Walfisz}, we get that
\begin{equation*}
E_2(x) \ll \frac1{\pi(x)} \cdot \frac{x}{(\log x)^{A}} \cdot z \ll \frac1{(\log x)^{A/2 - 1}}= o(1) ,
\end{equation*}
as $x \to +\infty$.
Observe that due to Lemma~\ref{lem:basic}\ref{ite:basic:upper}, if $d \leq y$ then $\ell(dk) \leq x^{1/2}$.
Hence, applying Lemma~\ref{lem:corollary-brun-titchmarsh} and Lemma~\ref{lem:bound-phi-ell}, we get that
\begin{equation*}
E_3(x) \ll  \frac1{\pi(x)} \cdot \frac{x}{\log x} \cdot \sum_{d \,>\, z} \frac1{\varphi\big(\ell(dk)\big)} = o(1) ,
\end{equation*}
as $x \to +\infty$.
Finally, using the trivial bound $\pi(x;b,1)\leq x/b$ and Lemma~\ref{lem:bound-phi-ell}, we get that 
\begin{align*}
    E_4(x) &\ll \frac1{\pi(x)} \sum_{d \,>\, y} \left(\frac{x}{\ell(dk)} + \frac{\pi(x)}{\varphi\big(\ell(dk)\big)}\right) \ll \frac{x}{\pi(x)} \sum_{d \,>\, y} \frac1{\varphi\big(\ell(dk)\big)} \\
    &\ll \frac{\log x \log \log y}{\exp\!\left(\delta (\log y)^{1/2} (\log \log y)^{1/2} \right)}
    = o(1) ,
\end{align*}
as $x \to +\infty$.
The proof is complete.
\end{proof}

By the definition of $\mathcal{R}_k$ and by the inclusion-exclusion principle, it follows easily that
\begin{equation*}
\#\big(\mathcal{P}_k \cap [1, x]\big) = \sum_{d \,\mid\, k} \mu(d) \, \#\big(\mathcal{R}_{dk}(x) \cap [1,x]\big)
\end{equation*}
for all $x > 0$.
Therefore, by Lemma~\ref{lem:Qkdens}, we get that
\begin{align}\label{equ:last}
\mathbf{r}(\mathcal{P}_k) &= \sum_{d \,\mid\, k} \mu(d) \, \mathbf{r}(\mathcal{R}_{dk}) = \sum_{d \,\mid\, k} \mu(d) \sum_{(e, \,dk) \,=\, 1} \frac{\mu(e)}{\varphi\big(\ell(dek)\big)} \nonumber \\
&= \sum_{d \,\mid\, k} \sum_{(e,\, k) \,=\, 1} \frac{\mu(de)}{\varphi\big(\ell(dek)\big)} = \sum_{f \,=\, 1}^\infty \frac{\mu(f)}{\varphi\big(\ell(fk)\big)} , 
\end{align}
since every squarefree integer $f$ can be written uniquely as $f=de,$ where $d$ and $e$ are squarefree integers such that $d\mid k$ and $\gcd(e,k)=1$. The rearrangement of series in~\eqref{equ:last} is justified by the absolute convergence of the series of Lemma~\ref{lem:Qkdens}. 
The proof is complete.

\section{Proof of the lower bound in~\texorpdfstring{\eqref{equ:A-bounds}}{(1)} and Proposition~\ref{prop:K-bounds}}

We need the following lemma.

\begin{lemma}\label{lem:AcapK}
Let $k$ be a positive integer such that $10 \mid k$ and $\bm{r}(\mathcal{P}_k) > 0$, and let $p \in \mathcal{P}_k$.
Then we have that $kp \in \mathcal{K}$.
\end{lemma}
\begin{proof}
Since $p \in \mathcal{P}_k$, we have that $\gcd(p - 1, F_{p - 1}) = k$.
Furthermore, since $5 \mid k$, we have that $p \equiv 1 \pmod 5$, and so $z(p)\mid p-1$ and $p \mid F_{p(p-1)}$ due to Lemma~\ref{lem:basic}\ref{ite:basic:zetap} and~\ref{ite:basic:zeta}.
In particular, $\gcd(p, F_{p(p-1)}) = p$.
For the sake of brevity, put $g := \gcd(p - 1, F_{p(p-1)})$.
We shall proved that $g = k$.
First, in light of Lemma~\ref{lem:basic}\ref{ite:basic:fibdiv}, we have that $k \mid g$.
Suppose that $q$ is a prime factor of $g / k$.
Then $q \neq p$ and $q \mid F_{p(p-1)} / F_{p-1}$.
Furthermore, by Lemma~\ref{lem:basic}\ref{ite:basic:zeta}, we have that $z(q) \mid p(p-1)$.
If $p \mid z(q)$ then, by Lemma~\ref{lem:basic}\ref{ite:basic:zetap}, $p \mid q - 1$, which is impossible since $q \leq p - 1$.
Thus $p \nmid z(q)$ and so $z(q)\mid p-1$. 
In particular, by Lemma~\ref{lem:basic}\ref{ite:basic:zeta}, we get that $q \mid F_{p-1}$.
Hence, Lemma~\ref{lem:basic}\ref{ite:basic:fibratio}, yields that $q = p$, which is impossible.
Therefore, we have that $g = k$.
Consequently, we get that 
\begin{equation*}
\gcd\!\big(p(p-1),F_{p(p-1)}\big) = \gcd(p-1,F_{p(p-1)}) \, \gcd(p,F_{p(p-1)}) = kp .
\end{equation*}
Thus $\mathcal{A}_{kp} \neq \varnothing$ and, by Lemma~\ref{lem:nonempty-Ak}, we have that $\gcd\!\big(\ell(kp),F_{\ell(kp)}\big)=kp$.
Also, since $2 \mid k$, we have that $2 \mid \ell(kp)$.
Hence, from Theorem~\ref{thm:existence} it follows that $kp \in \mathcal{K}$, as desired.
\end{proof}

Let us prove the lower bound of Proposition~\ref{prop:K-bounds}.
Note that $\ell(10)=30$ and $\gcd(\ell(10),F_{\ell(10)})=10$ so that, by Theorem~\ref{thm:existence}, we have that $\mathbf{r}(\mathcal{P}_{10})>0$.
Hence, applying Lemma~\ref{lem:AcapK} with $k = 10$, we get that
\begin{equation}\label{equ:K-below}
\#\big(\mathcal{K} \cap [1, x]\big) \gg \#\big\{kp : p \in \mathcal{P}_k \cap [1, x/k]\big\} \gg \frac{x}{\log x} , 
\end{equation}
which proves the lower bound.

If $k \in \mathcal{K}$ then, by Theorem~\ref{thm:existence}, we have that $\gcd\!\big(\ell(k),F_{\ell(k)}\big) = k$.
Hence, from~\cite[Lemma 2.2(iii)]{MR3859754}, it follows that $k$ belongs to $\mathcal{A}$.
Therefore $\mathcal{K} \subseteq \mathcal{A}$.
Consequently, on the one hand, by~\eqref{equ:K-below}, we get that
\begin{equation*}
    \#\big(\mathcal{A} \cap [1, x]\big) \geq \#\big(\mathcal{K} \cap [1, x]\big) \gg \frac{x}{\log x} ,
\end{equation*}
for all $x \geq 2$, which is the lower bound of~\eqref{equ:A-bounds}.
On the other hand, by Theorem~\ref{thm:A-bounds}, we get that
\begin{equation*}
\#\big(\mathcal{K} \cap [1, x]\big) \leq \#\big(\mathcal{A} \cap [1, x]\big) = o(x) ,
\end{equation*}
as $x \to +\infty$, which is the upper bound of Proposition~\ref{prop:K-bounds}.
The proofs are complete.

\bibliographystyle{amsplain-no-bysame}

\begin{thebibliography}{10}

\bibitem{MR2928495}
J.~J. Alba~Gonz\'{a}lez, F.~Luca, C.~Pomerance, and I.~E. Shparlinski, \emph{On
  numbers {$n$} dividing the {$n$}th term of a linear recurrence}, Proc. Edinb.
  Math. Soc. (2) \textbf{55} (2012), no.~2, 271--289.

\bibitem{MR1131414}
R.~Andr\'{e}-Jeannin, \emph{Divisibility of generalized {F}ibonacci and {L}ucas
  numbers by their subscripts}, Fibonacci Quart. \textbf{29} (1991), no.~4,
  364--366.

\bibitem{MR3690240}
A.~S. Chen, T.~A. Gassert, and K.~E. Stange, \emph{Index divisibility in
  dynamical sequences and cyclic orbits modulo {$p$}}, New York J. Math.
  \textbf{23} (2017), 1045--1063.

\bibitem{MR3251719}
P.~Cubre and J.~Rouse, \emph{Divisibility properties of the {F}ibonacci entry
  point}, Proc. Amer. Math. Soc. \textbf{142} (2014), no.~11, 3771--3785.

\bibitem{MR4067105}
T.~A. Gassert and M.~T. Urbanski, \emph{Index divisibility in the orbit of 0
  for integral polynomials}, Integers \textbf{20} (2020), Paper No. A16, 15.

\bibitem{MR2928585}
A.~Gottschlich, \emph{On positive integers {$n$} dividing the {$n$}th term of
  an elliptic divisibility sequence}, New York J. Math. \textbf{18} (2012),
  409--420.

\bibitem{MR2061214}
H.~Iwaniec and E.~Kowalski, \emph{Analytic number theory}, American
  Mathematical Society Colloquium Publications, vol.~53, American Mathematical
  Society, Providence, RI, 2004.

\bibitem{JhaPreprint}
A.~Jha, \emph{On terms in a dynamical divisibility sequence having a fixed
  {G.C.D.} with their index}, Preprint: \url{https://arxiv.org/abs/2105.06190}.

\bibitem{MR4017936}
S.~Kim, \emph{The density of the terms in an elliptic divisibility sequence
  having a fixed {G}.{C}.{D}. with their indices}, J. Number Theory
  \textbf{207} (2020), 22--41, With an appendix by M.~Ram~Murty.

\bibitem{MR3846400}
P.~Leonetti and C.~Sanna, \emph{A note on primes in certain residue classes},
  Int. J. Number Theory \textbf{14} (2018), no.~8, 2219--2223.

\bibitem{MR3859754}
P.~Leonetti and C.~Sanna, \emph{On the greatest common divisor of {$n$} and the
  {$n$}th {F}ibonacci number}, Rocky Mountain J. Math. \textbf{48} (2018),
  no.~4, 1191--1199.

\bibitem{MR3409327}
F.~Luca and E.~Tron, \emph{The distribution of self-{F}ibonacci divisors},
  Advances in the theory of numbers, Fields Inst. Commun., vol.~77, Fields
  Inst. Res. Math. Sci., Toronto, ON, 2015, pp.~149--158.

\bibitem{MR3983305}
D.~Mastrostefano, \emph{An upper bound for the moments of a gcd related to
  {L}ucas sequences}, Rocky Mountain J. Math. \textbf{49} (2019), no.~3,
  887--902.

\bibitem{MR3896876}
D.~Mastrostefano and C.~Sanna, \emph{On numbers {$n$} with polynomial image
  coprime with the {$n$}th term of a linear recurrence}, Bull. Aust. Math. Soc.
  \textbf{99} (2019), no.~1, 23--33.

\bibitem{MR360447}
H.~J.~A. Sall\'{e}, \emph{A maximum value for the rank of apparition of
  integers in recursive sequences}, Fibonacci Quart. \textbf{13} (1975),
  159--161.

\bibitem{MR3606950}
C.~Sanna, \emph{On numbers {$n$} dividing the {$n$}th term of a {L}ucas
  sequence}, Int. J. Number Theory \textbf{13} (2017), no.~3, 725--734.

\bibitem{MR3935356}
C.~Sanna, \emph{On numbers {$n$} relatively prime to the {$n$}th term of a
  linear recurrence}, Bull. Malays. Math. Sci. Soc. \textbf{42} (2019), no.~2,
  827--833.

\bibitem{MR3800663}
C.~Sanna and E.~Tron, \emph{The density of numbers {$n$} having a prescribed
  {G}.{C}.{D}. with the {$n$}th {F}ibonacci number}, Indag. Math. (N.S.)
  \textbf{29} (2018), no.~3, 972--980.

\bibitem{MR2747036}
J.~H. Silverman and K.~E. Stange, \emph{Terms in elliptic divisibility
  sequences divisible by their indices}, Acta Arith. \textbf{146} (2011),
  no.~4, 355--378.

\bibitem{MR1271392}
L.~Somer, \emph{Divisibility of terms in {L}ucas sequences by their
  subscripts}, Applications of {F}ibonacci numbers, {V}ol. 5 ({S}t. {A}ndrews,
  1992), Kluwer Acad. Publ., Dordrecht, 1993, pp.~515--525.

\bibitem{MR491445}
C.~L. Stewart, \emph{On divisors of {F}ermat, {F}ibonacci, {L}ucas, and
  {L}ehmer numbers}, Proc. London Math. Soc. (3) \textbf{35} (1977), no.~3,
  425--447.

\bibitem{MR3363366}
G.~Tenenbaum, \emph{Introduction to analytic and probabilistic number theory},
  third ed., Graduate Studies in Mathematics, vol. 163, American Mathematical
  Society, Providence, RI, 2015, Translated from the 2008 French edition by
  Patrick D. F. Ion.

\bibitem{MR4159096}
E.~Tron, \emph{The greatest common divisor of linear recurrences}, Rend. Semin.
  Mat. Univ. Politec. Torino \textbf{78} (2020), no.~1, 103--124.

\end{thebibliography}

\end{document}